\documentclass{amsart}

\usepackage{amssymb,amsmath,graphicx}

\usepackage{xcolor}

\newtoks\prt

\numberwithin{equation}{section}

\newtheorem{thm}{Theorem}[section]

\newtheorem{lemma}[thm]{Lemma}

\newtheorem{prop}[thm]{Proposition}

\newtheorem{cor}[thm]{Corollary}

\newtheorem*{claim}{Claim}

\theoremstyle{definition}

\newtheorem{definition}[thm]{Definition}

\def\eqn#1$$#2$${\begin{equation}\label#1#2\end{equation}}

\def\C{\mathcal C}

\def\E{\mathcal E}

\def\e{e^\ast}

\def\E{E^\ast}

\def\ep{\varepsilon}

\def\en{\mathbb N}

\def\er{\mathbb R}

\def \reg {\partial _{\kern1pt\text{reg}}}

\def\la{\langle}

\def\ra{\rangle}

\newcommand{\norm}[1]{\left\|#1\right\|}

\newcommand{\abs}[1]{\left| #1  \right|}

\begin{document}

\title[Isomorphisms of $\C(K)$ and cardinal invariants of derivatives of $K$]
{On isomorphisms of $\C(K)$ spaces and cardinal invariants of derivatives of $K$}

\author[J.~Rondo\v s]{Jakub Rondo\v s}
\address[J.~Rondo\v s]{Charles University\\
Faculty of Mathematics and Physics\\
Department of Mathematical Analysis \\
Sokolovsk\'{a} 83, 186 \ 75\\Praha 8, Czech Republic}

\email{jakub.rondos@gmail.com}



\thanks{}

\begin{abstract}
We present a necessary condition for a pair of $\C(K)$ spaces to be isomorphic in terms of topological properties of Cantor-Bendixon derivatives of $K$. This in particular gives a completely new information about the perfect kernels of such $K$. In the process, we extend known lower estimates of the Banach-Mazur distance between a pair of spaces of continuous functions from the case of scattered compact spaces to a more general setting. Next, we apply this general result to deduce some new information about isomorphisms of spaces of continuous functions over Eberlein compacta of height $\omega+1$. Further, we show that isomorphisms of $\C(K)$ spaces preserve the spread of $K$, and we also prove some new results for pairs of spaces of continuous functions whose Banach-Mazur distance is less than $3$.
\end{abstract}

\maketitle

\section{Introduction}

This work is devoted to the isomorphic theory of Banach spaces of continuous functions on compact Hausdorff spaces. More specifically, we are interested in finding out what properties of general compact spaces are preserved by isomorphisms of their spaces of continuous functions. Of course, some properties of compact spaces are maintained simply due to the fact that they can be equivalently reformulated as Banach space properties of the corresponding spaces of continuous functions. These include the properties of being scattered, being Eberlein, metrizability, and the value of weight, that correspond in terms of the $\C(K)$ space respectively to the properties of being Asplund, being weakly compactly generated, separability, and the value of density. Further, the cellularity of $K$ is preserved, since it can be expressed in terms of presence of copies of $c_0(\Gamma)$ in $\C(K)$, see \cite[page 230]{Rosenthal_L^infty}. An another proof of the preservation of cellularity follows as a special case of Theorem \ref{main} below. Also, the property of carrying a strictly positive measure is preserved, since $K$ carries a strictly positive measure if and only if there exists a weakly compact set in $\C(K)^{\ast}$ whose linear span is weak$^*$ dense, see \cite[Theorem 1.4]{Rosenthal_injective_C(S)}. The class of weakly Koszmider spaces is preserved, since this property of $K$ can be expressed using only commutators of operators on $\C(K)$, see \cite[Theorem 4.6]{SCHLACKOW_weakly_Koszmider}. There are more classes of compact spaces that are preserved due to the fact that they can be characterized by means of properties of $\C(K)$, see e.g. \cite[page 1099]{NEGREPONTIS_Banach_spaces_and_topology}.
Then there are properties that are preserved, even though they do not coincide with any Banach space property of the $\C(K)$ space, like the value of cardinality \cite{cengiz}.

On the other hand, there are known natural properties that are not preserved, that include the values of density (which can be derived from the fact that $l_{\infty}$ is isomorphic to $L_{\infty}([0, 1])$ by a theorem of Pelczynski), character (the countable power $K$ of the Alexandroff duplicate of the Cantor set is a first countable space whose $\C(K)$ space is isomorphic to a space of continuous functions over a compact space with uncountable character, since it is isomorphic to $\C(K) \oplus c_0(\omega_1)$, see \cite[comments before theorem 4.15]{Marciszewski_eberlein_problems} for more information about this compact space), the property of being Frechet-Urysohn \cite{Okunev_frechet}, or the topological dimension, which follows for example from the well-known Miljutin theorem \cite[Theorem 2.1]{RosenthalC(K)} which states that spaces of continuous functions over uncountable compact metric spaces are all isomorphic.
An interesting recent result in this direction is the one from \cite{Glodkowski_C(K)_dimension}. Here the author shows that it is consistent that there exists a $\C(K)$ space for which $L$ has the same dimension as $K$ whenever $\C(L)$ is isomorphic to $\C(K)$. It might be interesting to investigate this phenomenon also for other properties that are not preserved in general.
Further, it is a long-standing open problem whether the class of Corson compact spaces is preserved (e.g. \cite[Question 1]{KOSZMIDER_interplay_K_C(K)}). The answer is affirmative if the spaces are also homogenous \cite[Corollary 6.3]{Plebanek2015}, or under Martin’s axiom and the negation of continuum hypothesis \cite{Argyros_Mer_Neg_1988}. Of course, there are plenty of cardinal invariants and topological properties that one might consider, and for which it seems to be unknown whether they are preserved.

The main goal of this concrete paper is to prove that if two $\C(K)$ spaces are isomorphic then for a certain cardinal invariant, there is some kind of an asymptotic equality on the Cantor-Bendixon derivatives of the associated compact spaces, at each nonzero gamma number (we recall that a gamma number is either $0$ or a power of $\omega$). This in particular implies that the cardinal invariant must be the same on the perfect kernels. Of course, it is not in general possible to deduce some similarities between the Cantor-Bendixon derivatives of a given order $\alpha$ (the exception being $\alpha=0$), since it can easily happen that such a derived set is finite in one of the spaces, and it is infinite in the other space.

Despite the fact that all the presented results are new for the case or real-valued functions, we will prove them in a more general setting of Banach-space valued functions. Even though the structure of the spaces $\C(K, E)$ (and of isomorphisms of these spaces), where $E$ is a Banach space, is in general much more complicated than the structure of $\C(K)$, it has been found out in the last years that isomorphisms of $\C(K, E)$ spaces, where $E$ belongs to some suitable class of Banach spaces, behave very similarly as isomorphisms of $\C(K)$ spaces (in the sense of preserving the properties of $K$). There have been results in this direction proved, for example, for uniformly convex spaces (\cite{cambern-pacific}, \cite{Galego_classification}, \cite{GALEGO_Zahn_classification}) and spaces with nontrivial cotype (\cite{CandidoGalegoCotype}, \cite{candidogalego}, \cite{Galego_Zahn_dichotomy}). Later on, it has been found out that many of the result work also in a more general case of Banach spaces that do not contain an isomorphic copy of $c_0$ (\cite{CANDIDOc0}, \cite{GalegoVillamizar}, \cite{CandidoScattered}, \cite{rondos-somaglia}). 
Since for each pair of compact spaces $K_1, K_2$, we have the canonical isometries $\C(K_1, \C(K_2)) \simeq \C(K_1 \times K_2) \simeq \C(K_2, \C(K_1))$, it is clear that the assumptions on the Banach spaces $E$ cannot be plainly removed in these results.

To be able to state our results we first need to introduce several notions. To this end, we recall that a family $\mathcal{U}$ of nonempty subsets of a topological space $X$ has order $k \in \en$ if for every $x \in X$, there are at most $k$ sets from $\mathcal{U}$ that contain $x$. Further, we say that $\mathcal{U}$ is point-bounded if it has order $k$ for some $k \in \en$, and we say that $\mathcal{U}$ is point-finite if, for every $x \in X$, there are at most finitely many sets from $\mathcal{U}$ that contain $x$. The family $\mathcal{U}$ is $\sigma$-point-bounded if it is a union of countably many point-bounded families. Further, we say that the family $\mathcal{U}$ is \emph{meeting} a subset $A$ of $X$, if $A \cap U \neq \emptyset$ for each $U \in \mathcal{U}$, and we say that the family $\mathcal{U}$ is \emph{cellular} if sets from $\mathcal{U}$ are open and pairwise disjoint.

\begin{definition}
Let $X$ be a topological space and $A \subseteq X$. We define the \emph{relative cellularity} of $A$ in $X$ as
\[c(A, X)=\sup \{ \abs{\mathcal{U}}: \mathcal{U} \text{ is a cellular family in } X \text{ which is meeting } A\}.\]
Further, we consider the cardinal numbers
\begin{equation}
\nonumber
\begin{aligned}
c_{\text{pb}}(A, X)=&\sup \{ \abs{\mathcal{U}}: \mathcal{U} \text{ is a point-bounded} \text{ family of nonempty open subsets of } \\&X \text{ which is meeting } A\},
\end{aligned}
\end{equation}
\begin{equation}
\nonumber
\begin{aligned}
c_{\text{pf}}(A, X)=&\sup \{ \abs{\mathcal{U}}: \mathcal{U} \text{ is a point-finite} \text{ family of nonempty open subsets of } \\&X \text{ which is meeting } A\}.
\end{aligned}
\end{equation}
\end{definition}

We further denote $c(A)=c(A, A)$, if $A$ is considered as topological space with the topology inherited from $X$, and similarly for $c_{\text{pb}}(A)$ and $c_{\text{pf}}(A)$. Thus $c(A)$ is the classical cellularity of $A$, defined by
\[c(A)=\sup \{ \abs{\mathcal{U}}: \mathcal{U} \text{ is a cellular family in } A \},\]
Let us also recall the notion of spread:
\[s(A)=\sup \{ \abs{D}: D \subseteq A \text{ and } D \text{ is discrete in itself} \}.\]
It might be worth to note that these numbers are related by the formula
\[s(A)=\sup \{c(B): B \subset A\},\]
see e.g. \cite[page 37]{Arkhangelskii_cardinal_invariants}.
More information about the above-defined cardinal invariants will be given in the next section. 

Our main result is the following one. It provides a generalization of the known lower estimates of the Banach-Mazur distance between a couple of spaces of continuous functions, that have been subsequently proved in the last years by several authors, see e.g. the papers \cite{CandidoGalegoComega}, \cite{CANDIDOc0}, \cite{rondos-scattered-subspaces} and \cite{rondos-somaglia}. Even though the constant estimating the Banach-Mazur distance remains the same, there is still a significant improvement. It lies in the fact that so far, the estimates only worked for scattered spaces and were based on their heights, while in the following result, the estimate works for spaces that need not be scattered and is based on topological properties of their derived sets.  

\begin{thm}
\label{core}
Let $K_1, K_2$ be nonempty compact spaces, $E$ be a Banach space and suppose that $T:\C(K_1) \rightarrow \C(K_2, E)$ is an isomorphic embedding. Let $\alpha$ be an ordinal and $n, k \in \en$, $n \geq k$ be such that the set $K_2^{(\omega^{\alpha}k)}$ is infinite.
Moreover, let at least one of the following conditions hold. 
\begin{itemize}
    \item[(i)] $E$ has nontrivial cotype and $c_{\text{pb}}(K_1^{(\omega^{\alpha}n)}, K_1)>c_{\text{pb}}(K_2^{(\omega^{\alpha}k)}, K_2)$,
    \item[(ii)] $E$ does not contain $c_0$ isomorphically and $c_{\text{pb}}(K_1^{(\omega^{\alpha}n)}, K_1)>c_{\text{pf}}(K_2^{(\omega^{\alpha}k)}, K_2)$.
\end{itemize}
	Then $\norm{T}\norm{T^{-1}}\geq \max \{3, \frac{2n+2-k}{k}\}$.
\end{thm}

From this theorem we obtain the following result which concerns the asymptotic equality of the above-defined cardinal numbers of the derivatives at each nonzero gamma number. In particular, we obtain completely new information about perfect kernels of the respective compact spaces $K$, denoted as $K^{(\infty)}$. 

\begin{thm}
\label{main}
Let $K_1, K_2$ be infinite compact spaces and $E$ be a Banach space such that $\C(K_1)$ is isomorphic to a subspace of $\C(K_2, E)$. Then for each ordinal $\alpha$:
\begin{itemize}
    \item[(i)] If $E$ has nontrivial cotype, then 
    \[ \min_{\beta<\omega^{\alpha}} c_{\text{pb}}(K_1^{(\beta)}, K_1) \leq \min_{\beta<\omega^{\alpha}} c_{\text{pb}}(K_2^{(\beta)}, K_2) .\]
    In particular, $c(K_1) \leq c(K_2)$ and $c_{\text{pb}}(K_1^{(\infty)}, K_1) \leq c_{\text{pb}}(K_2^{(\infty)}, K_2)$.
    \item[(ii)] If $E$ does not contain an isomorphic copy of $c_0$, then 
    \[ \min_{\beta<\omega^{\alpha}} c_{\text{pb}}(K_1^{(\beta)}, K_1) \leq \min_{\beta<\omega^{\alpha}} c_{\text{pf}}(K_2^{(\beta)}, K_2) .\]
    In particular, $c(K_1) \leq c(K_2)$ and $c_{\text{pb}}(K_1^{(\infty)}, K_1) \leq c_{\text{pf}}(K_2^{(\infty)}, K_2)$.
\end{itemize}
\end{thm}

To make the statement of the theorem more understandable, we note that for each compact space $L$, $c(L)=c_{\text{pb}}(L)=c_{\text{pf}}(L)$, which will be shown in Lemma \ref{rovnost_c=c_pf}(iii) below. Further, the reason why the minima in the above result exist is due to the fact that for each compact space $K$ and ordinal $\alpha$, the nets of cardinal numbers $\{c_{\text{pb}}(K^{(\beta)}, K)\}_{\beta < \omega^{\alpha}}$ and $\{c_{\text{pf}}(K^{(\beta)}, K)\}_{\beta < \omega^{\alpha}}$ are nonincreasing, hence eventually constant. For nonzero ordinals $\alpha$, the above minima might be equivalently expressed in terms of limits, where $\beta \rightarrow \omega^{\alpha}$.

Further, we note that while items (ii) in the above results might seem a bit unnatural, since one compares two different cardinal invariants, we decided to include them, as it might indicate an interesting interaction of the geometry of the Banach space $E$ with the underlying compact spaces. Moreover, in some cases, the numbers $c_{\text{pb}}$ and $c_{\text{pf}}$ coincide, as we show below. We do not know whether the conclusion of the items (i) is true in the case when $E$ does not contain a copy of $c_0$. 

Next we note that Theorem \ref{main} serves as a common generalization of several known facts. Firstly, we get a new proof of the preservation of cellularity of $K$ by isomoprhisms of $\C(K, E)$ spaces, where $E$ does not contain an isomorphic copy of $c_0$. Even though we have not found any exact reference for this, it can be derived from two published results, namely, from \cite[page 230]{Rosenthal_L^infty} and \cite[Theorem 3.1]{Galego-Hagler_kopie_c0}, that concern respectively the relation of cellularity of $K$ and the presence of copies of $c_0(\Gamma)$ in $\C(K)$, and the question of presence of copies of $c_0(\Gamma)$ in $\C(K)$ and $\C(K, E)$ when $E$ does not contain $c_0$. Secondly, the result generalizes the fact that if $\C(K_1)$ is isomorphic to a subspace of $\C(K_2, E)$ and $E$ does not contain $c_0$, then the Szlenk index of $\C(K_1)$ is less than or equal to the Szlenk index of $\C(K_2)$. This for $E=\er$ follows simply from the fact the Szlenk index of a Banach space cannot be increased by isomorphic embeddings, and the case for a general Banach which does not contain $c_0$ has been recently proved in \cite{rondos-somaglia}. 
Indeed, the statement that the Szlenk index of $\C(K_1)$ is less than the Szlenk index of $\C(K_2)$ might be equivalently reformulated by saying that for each ordinal $\alpha$, if $\min_{\beta<\omega^{\alpha}} c_{\text{pf}}(K_2^{(\beta)}, K_2)=0$ then 
$\min_{\beta<\omega^{\alpha}} c_{\text{pb}}(K_1^{(\beta)}, K_1)=0$. To see this, one needs to use the result from \cite{CAUSEY_C(K)_index} which shows that the Szlenk index of $\C(K)$ can be computed from the height of $K$. 

Since for every Banach space $E$, the space $\C(K, E)$ contains a canonical isometric copy of $\C(K)$, the following is an immediate consequence of Theorem \ref{main}(i).

\begin{cor}
\label{cor-main}
Let $K_1, K_2$ be infinite compact spaces and $E_1, E_2$ be Banach spaces with a nontrivial cotype such that $\C(K_1, E_1)$ is isomorphic to $\C(K_2, E_2)$. 
Then for each ordinal $\alpha$,
\[ \min_{\beta<\omega^{\alpha}} c_{\text{pb}}(K_1^{(\beta)}, K_1)=\min_{\beta<\omega^{\alpha}} c_{\text{pb}}(K_2^{(\beta)}, K_2) .\]
In particular, $c(K_1)=c(K_2)$ and $c_{\text{pb}}(K_1^{(\infty)}, K_1)=c_{\text{pb}}(K_2^{(\infty)}, K_2)$.
\end{cor}

Of course, isomorphisms of $\C(K)$ spaces do not preserve the property of being a perfect compact space, which is easily seen from the Miljutin theorem. However, as a corollary of our results we obtain the information that if $K_1$ is perfect and $\C(K_1)$ is isomorphic to $\C(K_2)$, then the perfect kernel of $K_2$ must be large in $K_2$, in the sense of cellularity. 

\begin{cor}
\label{prefect}
Let $K_1, K_2$ be infinite compact spaces and $E_1, E_2$ be Banach spaces with a nontrivial cotype. Assume that $K_1$ is perfect and $\C(K_1, E_1)$ is isomorphic to $\C(K_2, E_2)$. Then $c_{\text{pb}}(K_2^{(\infty)}, K_2)=c(K_2)$. Consequently, $c(K_2^{(\infty)}) \geq c(K_2)$.
\end{cor}

\begin{proof}
By Corollary \ref{cor-main}, $c_{\text{pb}}(K_2^{(\infty)}, K_2)=c_{\text{pb}}(K_1^{(\infty)}, K_1)=c_{\text{pb}}(K_1, K_1)=c_{\text{pb}}(K_1)=c(K_1)$. Moreover, by Theorem \ref{main}(i) for $\alpha=0$, $c(K_1)=c(K_2)$. Hence $c(K_2) =c_{\text{pb}}(K_2^{(\infty)}, K_2) \leq c_{\text{pb}}(K_2^{(\infty)}, K_2^{(\infty)})=c_{\text{pb}}(K_2^{(\infty)}) = c(K_2^{(\infty)})$.
\end{proof}

Next, we apply Theorem \ref{main} to the particular case when the considered compact spaces are Eberlein and of height $\omega+1$. The next result follows immediately from Theorem \ref{main}(ii) and the fact that in the case of an Eberlein compact space $K$ of height $\omega+1$, $c_{\text{pb}}(K^{(n)}, K)=c_{\text{pf}}(K^{(n)}, K)=\abs{K^{(n)}}$ for each $n \in \en$, see Lemma \ref{cel-Eberlein} below.

\begin{thm}
\label{main-Eberlein}
Let $K_1, K_2$ be Eberlein compact spaces of height $\omega+1$ and $E$ be a Banach space not containing an isomorphic copy of $c_0$. If $\C(K_1)$ is isomorphic to a subspace of $\C(K_2, E)$, then  
    \[ \lim_{n \rightarrow \infty} \abs{K_1^{(n)}} \leq \lim_{n \rightarrow \infty} \abs{K_2^{(n)}}.\]
\end{thm}

Scattered Eberlein compacta form a nice and well-behaved class of compact spaces that has been studied closely in the last years, see e.g. \cite{Bell2007OnSE}, \cite{Bell_Marciszewski_univeral_eberlein} or \cite{Marciszewski_eberlein_problems}. Properties of spaces of continuous functions on this class of compact spaces were treated e.g. in \cite{Godefroy2000SubspacesO} and \cite{Marciszewski2003OnBS}. In particular, it follows from \cite[Theorem 4.8]{Godefroy2000SubspacesO} and \cite{cengiz} that if $K_1, K_2$ are Eberlein compacta of finite height and weight less then $\omega_{\omega}$, then the spaces $\C(K_1), \C(K_2)$ are isomorphic if and only if $\abs{K_1}=\abs{K_2}$ (to avoid confusion we note that while in the statement of \cite[Theorem 4.8]{Godefroy2000SubspacesO} it is assumed that weight of the considered compact space is $\omega_1$, the authors remark on page 800 that their results can be extended with similar proofs to the case when the weight is less than $\omega_{\omega}$, and we also recall that the weight of a scattered Eberlein compactum is equal to its cardinality). From this, it might be tempting to speculate that the isomorphism classes of spaces of continuous functions over scattered Eberlein compacta (at least those of weight less than $\omega_{\omega}$) are perhaps determined just by height and cardinality of the compact spaces. Theorem \ref{main-Eberlein} shows that this is not the case, because for Eberlein compacta of height $\omega+1$, the situation depends also on how the cardinality is distributed among the derived sets. Of course, it is also natural to wonder whether there is a similar behaviour for scattered Eberlein compact space of higher heights. Concerning this question, we note that Eberlein compacta of height $\omega+1$ are uniform Eberlein, see \cite{Bell2007OnSE}, while Eberlein compacta of height $\omega+2$ need not be, see \cite{BenyaminiStarbird_UniformEberlein} and \cite{Marciszewski_sequentiall}. Thus perhaps one should rather consider the class of scattered uniform Eberlein compact spaces instead. 

Further, it is a very interesting question whether the implication can be reversed, that is, whether one can describe the isomorphism classes of scattered (uniform) Eberlein compact spaces (perhaps just of weight less than $\omega_{\omega}$) only using the height and cardinalities of derived sets. This would provide a complete isomorphic classification of spaces of continuous functions on a natural and well-known class of nonmetrizable compacta. So far, outside of the class of metrizable spaces, the isomorphic classification of $\C(K)$ spaces is known only for the class of ordinal intervals \cite{Kislyakov_classification}, for the class of compact groups \cite{Pelczynski_book_1968}, for Eberlein compacta of finite height and weight less than $\omega_{\omega}$, as explained above, which has been generalized in a certain way in \cite{Aviles_classification}, or in some special cases, see \cite{Galego_classification}. Also, projective tensor products of $\C(K)$ spaces over countable compact spaces were recently classified \cite{GalegoCauseySamuel_projective_classification}.

Of course, there are other natural questions and open problems related to the above results. The most obvious of these is to try to replace the cardinal function $c_{\text{pb}}$ in Theorem \ref{main} by other cardinal invariants, if not in full generality, then at least in some classes of compact spaces, as in Theorem \ref{main-Eberlein}. For example, the result from \cite{Gordon3} states that if a pair of $\C(K)$ spaces is isomorphic by an isomorphism with distortion strictly less than $3$, then all derivatives of these compact spaces have the same cardinality. Also, we recall that the cardinality of the entire compact space is preserved by isomorphisms of its space of continuous functions, see \cite{cengiz}. Generalizations to the vector-valued case were given in \cite{GalegoVillamizar}. These results indicate it is not out of question that Theorem \ref{main} might still be true if we replace the cardinal function $c_{\text{pb}}$ by cardinality. We note that for metrizable compact spaces it is true, but here the reason for this is very specific. Indeed, if $K$ is a metrizable compact space, then it is either uncountable, in which case all of its derivatives have cardinality of the continuum (because they contain a copy of the Cantor space), or it is countable, hence scattered, and then the desired inequality follows for example from \cite{rondos-somaglia}.
Also, it might be possible to replace the function $c_{\text{pb}}$ by weight or spread in Theorem \ref{main}. On the other hand, it is not possible to obtain an analogy of Theorem \ref{main} for density or character, for general compact spaces. Indeed, as mentioned above, the equality of the minima does not in general hold even when $\alpha=0$.

As a subsequent result we prove that isomorphic embeddings of $\C(K)$ spaces preserve the spread of $K$. Even though the value of the spread of $K$ may be expressed by means of certain biorthogonal systems in $\C(K)$, see \cite[Proposition 3.2]{Koszmider_biorthogonal}, this does not seem to imply the desired information, at least not in some obvious way. Moreover, we will work again with vector-valued functions. More specifically, we prove the following result. We do not know whether the result holds under the weaker assumption that $E$ does not contain $c_0$.

\begin{thm}
\label{spread}
Let $K_1, K_2$ be compact Hausdorff spaces and $E$ be a Banach space with nontrivial cotype. If $\C(K_1)$ is isomorphically embedded into $\C(K_2, E)$ and $s(K_1)$ is infinite, then $s(K_1) \leq s(K_2)$.
\end{thm}

Our last result is devoted to isomorphisms that have distortion strictly less than $3$. We recall that the well-known Amir-Cambern theorem says that the Banach-Mazur distance of a pair of $\C(K)$ spaces is at least $2$, if the associated compact spaces are not homeomorphic. It has been thought for some time that the number $2$ might possibly be replaced by $3$ in this result (see \cite{amir}). Even though it has been shown that this is not possible, as the number $2$ is optimal \cite{cohen-bound2}, it has been proved in \cite{Gordon3} that if the Banach-Mazur distance between $\C(K_1)$ and $\C(K_2)$ is strictly less than $3$, then the compact spaces $K_1$, $K_2$ must still be very similar, at least in some sense. Namely, it holds that for each ordinal $\alpha$, the cardinalities of the derivatives $K_1^{(\alpha)}$ and $K_2^{(\alpha)}$ are equal. This has been generalized to the case of functions with values in a Banach space with nontrivial cotype in \cite{CandidoGalego3}, and to the case of functions with values in a Banach space that does not contain $c_0$ in \cite{GalegoVillamizar}.
We show that in this case, the derived sets must share more common properties.

\begin{thm}
\label{dist3}
Let $K_1, K_2$ be compact Hausdorff spaces and $E$ be a Banach space. Assume that there exists an isomorphic embedding $T:\C(K_1) \rightarrow \C(K_2, E)$ with $\norm{T}\norm{T^{-1}}<3$. Then for each ordinal $\alpha$ such that $K_1^{(\alpha)}$ is infinite it holds:
\begin{itemize}
    \item[(i)] If $E$ has nontrivial cotype, then $c_{\text{pb}}(K_1^{(\alpha)}, K_1) \leq c_{\text{pb}}(K_2^{(\alpha)}, K_2)$ and $s(K_1^{(\alpha)}) \leq s(K_2^{(\alpha)})$.
    \item[(ii)] If $E$ does not contain $c_0$ isomorphically, then $c_{\text{pb}}(K_1^{(\alpha)}, K_1) \leq c_{\text{pf}}(K_2^{(\alpha)}, K_2)$.
\end{itemize}
\end{thm}

\section{Definitions and preliminary results}

In this section, we first collect several notions and conventions. To start with, we stress that all topological spaces are assumed to be Hausdorff. Next, for a compact space $K$ and a Banach space $E$, let $\C(K, E)$ denote the space of all continuous $E$-valued functions endowed with the sup-norm. We write $\C(K)$ for $\C(K, \er)$. For a function $f \in \C(K)$, $\text{spt} f$ stands for the \emph{support} of $f$, that is, the closure of the set $\{ x \in K: f(x) \neq 0\}$.

Further, we recall that the \emph{derivative} of a topological space $X$ is defined recursively as follows. The set $X^{(1)}$ is the set of accumulation points of $X$, and for an ordinal $\alpha>1$, let $X^{(\alpha)}=(X^{\beta})^{(1)}$, if $\alpha=\beta+1$, and $X^{(\alpha)}=\bigcap_{\beta<\alpha} X^{(\beta)}$ when $\alpha$ is a limit ordinal. Moreover, let $X^{(0)}=X$. It is well-known that there exists an ordinal $\alpha$ such that $X^{(\alpha)}=X^{(\beta)}$ for each $\beta>\alpha$. For such an ordinal $\alpha$, we denote $X^{(\infty)}=X^{(\alpha)}$. The set $X^{(\infty)}$ is called the \emph{perfect kernel} of $X$. The space is called \emph{perfect} if $X=X^{(\infty)}$, that is, if $X$ has no isolated points.
The topological space $X$ is called \emph{scattered} if its perfect kernel is empty, and the minimal $\alpha$ such that $X^{(\alpha)}$ is empty is called the \emph{height} of $X$ and is denoted by $ht(X)$. If $X$ is not scattered, then we define $ht(X)$ to be $\infty$. It is easy to see that if $K$ is a scattered compact space, then $ht(K)$ is always a successor ordinal.
Further, we recall that a gamma number is an ordinal which is either $0$ or of the form $\omega^{\alpha}$ for some ordinal $\alpha$. Given an ordinal $\alpha$, we let $\Gamma(\alpha)$ denote the minimum gamma number which is not less than $\alpha$. This minimum exists due to the fact that $\omega^{\alpha} \geq \alpha$ for any ordinal $\alpha$. For completeness we set $\Gamma(\infty)=\infty$, and we note that we use the convention that $\alpha<\infty$ for each ordinal $\alpha$. It has been known since the fundamental paper of Bessaga and Pelczynski \cite{BessagaPelcynski_classification} that gamma numbers play a crucial role in the isomorphic theory of $\C(K)$ spaces. Actually, they play a crucial role in the isomorphic theory of general Banach spaces, since the Slzenk index of each Banach space, if it is bounded, is always a gamma number, see e.g. \cite[Proposition 3.3]{Lancien2006}. It is therefore not surprising that they appear also in our results.

 We further recall that a compact space is \emph{Eberlein} if it is homeomorphic to a weakly compact subset of a Banach space, and it is \emph{uniform Eberlein} if it is homeomorphic to a weakly compact subset of a Hilbert space. We recall that a Banach space $E$ has \emph{nontrivial cotype} if there exists a real number $q \geq 2$ and a constant $C>0$ such that for each $n \in \en$ and $e_1, \ldots, e_n \in E$, 
\[ (\sum_{i=1}^n \norm{e_i}^q)^{\frac{1}{q}} \leq C(\int_{0}^1 \norm{\sum_{i=1}^n r_i(t) e_i}^2 dt)^{\frac{1}{2}},\]
where $r_i:[0, 1] \rightarrow \er$ are the Rademacher functions, defined for $t \in [0, 1]$ by $r_i(t)=\text{sign } (\sin(2^i \pi t))$. It is well-known that if a Banach space has nontrivial cotype, then it does not contain an isomorphic copy of $c_0$. If $E_1, E_2$ are Banach spaces and $T: E_1 \rightarrow E_2$ is an isomorphic embedding, then the \emph{distortion} of $T$ is the number $\norm{T}\norm{T^{-1}}$. The \emph{Banach-Mazur} distance between $E_1$ and $E_2$ is defined to be the infimum of distortions of all surjective isomorphisms of $E_1$ and $E_2$, if the spaces are isomorphic, otherwise it is $\infty$. 

Now, we collect some information about the cardinal numbers $c(A, X)$, $c_{\text{pb}}(A, X)$ and $c_{\text{pf}}(A, X)$, that were defined in the introduction. The advantage of these relative invariants compared to the regular notion of cellularity is that they carry some information about the way $A$ is topologically situated in $X$. Moreover, unlike classical cellularity, they are monotone, meaning that if $A \subseteq B \subseteq X$, then $c(A, X) \leq c(B, X)$, $c_{\text{pb}}(A, X) \leq c_{\text{pb}}(B, X)$ and $c_{\text{pf}}(A, X) \leq c_{\text{pf}}(B, X)$. The concept of relative cellularity has been studied in the context of topological groups, see e.g. \cite{PASYNKOV_relative_cellularity} or \cite{gartside-maltsevretral}. We have not found any reference for the numbers $c_{\text{pb}}(A, X)$ and $c_{\text{pf}}(A, X)$. Of course, point-bounded or point-finite families of open sets are used quite frequently, but probably not so much in this relative setting. Further, it is clear that $c(A, X) \leq c(A)$ for each subset of $A$ of $X$, and similarly for $c_{\text{pb}}$ and $c_{\text{pf}}$. A natural example where this inequality is strict is the case when $X=\beta \omega$ and $A=\beta \omega \setminus \omega$. Indeed, it is well-known that $\beta \omega$ is separable, hence 
\[c(\beta \omega \setminus \omega, \beta \omega) \leq c(\beta \omega) \leq d(\beta \omega)=\omega,\]
where $d$ stands for density. On the other hand, the cellularity of $\beta \omega \setminus \omega$ is uncountable. This follows from the fact that, any almost disjoint family $\mathcal{P}$ of subsets of $\omega$ induces a pairwise disjoint family of nonempty open subsets of $\beta \omega \setminus \omega$, and it is a standard fact that there exist uncountable almost disjoint families of natural numbers. 

Further, it is natural to ask for which subsets $A$ of a general compact space $K$, some of the numbers $c(A, K)$, $c_{\text{pb}}(A, K)$ and $c_{\text{pf}}(A, K)$ coincide.
While it is well-known that the equality $c(A)=c_{\text{pb}}(A)$ holds for each topological space $A$, the equality $c(A, K)=c_{\text{pb}}(A, K)$ for a general subset $A$ of $K$ seems to be unclear.
It was proved in \cite[Lemma 1.2]{Rosenthal_injective_C(S)} (see also \cite[Lemma 2.2]{TODORCEVIC_chain}, \cite[Lemma 4.2]{Rosenthal_L^infty} or \cite[Theorem 3.8]{TALL_chain_condition} for slightly different proofs) using the Baire category theorem, that in a compact ccc space, point-finite families of open sets must be countable. Thus one might expect that if $\mathcal{U}$ is a point-finite family of open sets which is meeting a subset $A$ of a compact space $K$, then $\abs{\mathcal{U}} \leq c(A, K)$. While we do not know whether this holds or not, the direct modification of any of the above proofs to the relative setting and to the case of arbitrary large cardinalities, only seems to work when the set $A$ is open, or, more generally, when $A$ is contained in the closure of its interior. If this were true also when the set $A$ is compact, most of the results of this paper for Banach spaces $E$ that do not contain $c_0$ and those with nontrivial cotype would be united.

\begin{lemma}
\label{rovnost_c=c_pf}
Let $K$ be a compact space and $A$ is an infinite subset of $K$.
\begin{itemize}
    \item[(i)] It holds $c(A, K) \leq c(A) \leq s(A)$, and
    \[c(A, K) \leq c_{\text{pb}}(A, K) \leq c_{\text{pf}}(A, K) \leq \min \{c_{\text{pf}}(A), c(K)\}.\]
    \item[(ii)] It holds $c(A)=c(\overline{A})$, $c(A, K)=c(\overline{A}, K)$, $c_{\text{pb}}(A, K)=c_{\text{pb}}(\overline{A}, K)$, and $c_{\text{pf}}(A, K)=c_{\text{pf}}(\overline{A}, K)$.
    \item[(iii)] If $A$ is open, then $c(A)=c(A, K)=c_{\text{pf}}(A, K)=c_{\text{pf}}(A)$.
    \item[(iv)] If $A \subseteq \overline{int(A)}$, then $c(A, K)=c_{\text{pf}}(A, K)$.
    \item[(v)] If $A$ is dense in $K$, then $c(A, K)=c(A)=c(K)$.
\end{itemize}
\end{lemma}

\begin{proof}
All the inequalities in (i) are elementary except for $c_{\text{pf}}(A, K) \leq c(K)$. However, it is trivial that $c_{\text{pf}}(A, K) \leq c_{\text{pf}}(K)$, and it follows from (iii) that $c_{\text{pf}}(K)=c(K)$.

Further, the equality $c(A)=c(\overline{A})$ follows due to the observation that for a family $\mathcal{U}$ of open subsets of $K$ it holds that the family $\{U \cap \overline{A}: U \in \mathcal{U}\}$ is cellular in $\overline{A}$ if and only if the family $\{U \cap A: U \in \mathcal{U}\}$ is cellular in $A$. 
The rest of the proof of (ii) is immediate, since the closure of $A$ intersects the same open subsets of $K$ as $A$ itself.

For the proof of (iii), let $A$ is open. It is simple to realize that $c(A)=c(A, K)$ and $c_{\text{pf}}(A, K)=c_{\text{pf}}(A)$. Thus, we assume that there is a point-finite family $\mathcal{U}$ of nonempty open sets which is meeting $A$, of cardinality strictly greater than $c(A, K)$, and we seek a contradiction. For each $n \in \en$ we consider the sets 
\begin{equation}
\nonumber
\begin{aligned}
&G_n=\{x \in A: x \text{ is contained in at most } n \text{ distinct sets from } \mathcal{U}\}, \\&
\mathcal{U}_n=\{U \in \mathcal{U}:\text{int} G_n \cap U\neq \emptyset\}.
\end{aligned}
\end{equation}
The main part of the proof relies in showing that for each $n \in \en$, the set $\mathcal{U}_n$ has cardinality at most $c(A, K)$. 

To this end, we find a maximal cellular collection $\mathcal{C}_n$ of open subsets of $A$, each intersecting at most $n$ sets from $\mathcal{U}$. The cardinality of such a collection is at most $c(A, K)$. We claim that each member of the system $\mathcal{U}_n$ intersects a member of $\mathcal{C}_n$. Thus, we assume that there exists some $V \in \mathcal{U}_n$ that does not intersect any member of $\mathcal{C}_n$ and we seek a contradiction. Let $n_0$ stand for the maximal $k \in \{0, \ldots, n-1\}$ such that there exist distinct sets $U_{1}, \ldots U_k$ in the system $\mathcal{U}$ such that the intersection
\[\text{int} G_n \cap V \cap U_1 \cap \ldots \cap U_k\]
is nonempty.
We choose $U_1, \ldots, U_{n_0} \in \mathcal{U}$ such that the intersection $ \text{int} G_n \cap V\cap U_1 \cap \ldots \cap U_{n_0}$ is nonempty (if $n_0=0$, we do not select any sets $U_i$). But then, the set $\text{int} G_n \cap V \cap U_1 \cap \ldots \cap U_{n_0}$ is a nonempty open subset of $A$ which intersect exactly $n_0+1 \leq n$ sets from $\mathcal{U}$ (namely, it intersects only the sets $V, U_1, \ldots, U_{n_0}$), but it does not meet $\mathcal{C}_n$, which contradicts the maximality of the system $\mathcal{C}_n$. Thus each member of the system $\mathcal{U}_n$ intersects a member of $\mathcal{C}_n$. On the other hand, by the definition of $\C_n$, each member of $\C_n$ intersects at most $n$ members of $\mathcal{U}_n$. Hence $\abs{\mathcal{U}_n} \leq \abs{\mathcal{C}_n} \leq c(A, K)$.

Further, by our assumption, $\abs{\mathcal{U}}>c(A, K)$, hence it follows that the set $\mathcal{U} \setminus \bigcup_{n=1}^{\infty} \mathcal{U}_n$ is nonempty, that is, there exists a set $U \in \mathcal{U}$ which does not intersect $\bigcup_{n=1}^{\infty} \text{int } G_n$. On the other hand, it is clear that $U \cap A \subseteq \bigcup_{n=1}^{\infty} G_n$. Moreover, since each $G_n$ is closed in $A$, the set $G_n \setminus \text{int } G_n$ is a nowhere dense subset of $A$. Thus the nonempty open set $U \cap A$ is contained in the meager set $\bigcup_{n=1}^{\infty} (G_n \setminus \text{int } G_n)$, which contradicts the Baire category theorem. 

Further, the item (iv) follows by combination of (ii) and (iii), since by the monotonicity we have $c(int A, K) \leq c(A, K) \leq c(\overline{int(A)}, K)$ and 
$c_{\text{pf}}(int A, K) \leq c_{\text{pf}}(A, K) \leq c_{\text{pf}}(\overline{int(A)}, K)$, and hence by (ii) and (iii), all these cardinals are the same.

Finally, (v) follows immediately by (ii). 
\end{proof}

In general, it seems to be a difficult task to try to identify the above-defined relative cardinal invariants with some more familiar cardinal numbers, in some concrete classes of compact spaces. One example where this is possible is the following.

\begin{lemma}
\label{cel-Eberlein}
Let $K$ be a scattered Eberlein compact space and $A$ be an infinite subset of $K$. Then:
\begin{itemize}
    \item[(i)] It holds $c_{\text{pf}}(A, K)=\abs{A}$.
    \item[(ii)] If height of $K$ is $\omega+1$, then $c_{\text{pb}}(A, K)=\abs{A}$.
    \end{itemize}
\end{lemma}

\begin{proof}
Since $A$ is infinite, it clearly holds $c_{\text{pb}}(A, K) \leq c_{\text{pf}}(A, K) \leq \abs{A}$.

On the other hand, since $K$ is a scattered Eberlein compactum, by \cite[Lemma 1.1]{Bell2007OnSE}, there
exists a point-finite family $\{U_x: x \in K\}$ such that each $U_x$ is a neighbourhood of $x$. Thus the family $\{U_x: x \in A\}$ is point-finite as well, and it is clearly meeting $A$. Hence $c_{\text{pf}}(A, K) \geq \abs{A}$.

For the proof of the second statement, if height of $K$ is $\omega+1$, by the proof of \cite[Theorem 0.1]{Bell2007OnSE}, there exists a $\sigma$-point-bounded family $\{U_x : x \in K\}$ of distinct open sets in $K$, such that each $U_x$ contains $x$. Thus the family $\{ U_x : x \in A\}$ is also $\sigma$-point-bounded in $K$. Thus, it consists of countably many point-bounded families, and each of them is clearly meeting $A$. Thus the cardinality of $\{ U_x : x \in A\}$ is at most $c_{\text{pb}}(A, K)$. Hence $\abs{A} \leq c_{\text{pb}}(A, K)$.
\end{proof}

\section{Proofs of main results}

In this section we prove our main results. To make the proofs shorter and more transparent, we introduce the following notion.

\begin{definition}
Let $K$ be a compact space and $n \in \en$. We say that a system $\mathcal{P}$ consisting of ordered $n$-tuples of distinct functions from $\C(K)$ is a \emph{point-bounded system of bumps of height} $n$ if  the following assertions hold:
\begin{itemize}
    \item for each $(f_1, \ldots, f_{n})  \in \mathcal{P}$, $1 \geq f_1 \geq \ldots f_{n} \geq 0$, 
    \item for each $(f_1, \ldots, f_{n})  \in \mathcal{P}$, the set $\{ x \in K: f_n(x)=1 \}$ has nonempty interior in $K$,
    \item the system $\{\text{spt } f_1\}_{(f_1, \ldots, f_n) \in \mathcal{P}}$ is point-bounded.
\end{itemize}
The \emph{order} of $\mathcal{P}$ is the order of the family $\{\text{spt } f_1\}_{(f_1, \ldots, f_n) \in \mathcal{P}}$. The system $\mathcal{P}$ is \emph{cellular} if it has order $1$.
A \emph{point-bounded system of bumps} is a point-bounded system of bumps of height $1$.
Further, we say that $\mathcal{P}$ is \emph{meeting} a set $A \subset K$ if for each $(f_1, \ldots, f_{n}) \in \mathcal{P}$ and $i=1, \ldots n$, the set $int(\{ x \in K: f_i(x)=1 \})$ intersects $A$. 
\end{definition}

It is clear that if $m<n$ and $\mathcal{P}$ is a point-bounded system of bumps of height $n$, then $\{(f_1, \ldots ,f_m): (f_1, \ldots ,f_n) \in \mathcal{P} \}$ is a point-bounded system of bumps of height $m$. Moreover, any subset of $\mathcal{P}$ is a point-bounded system of bumps of height $n$. Further, a simple usage of Urysohn´s lemma shows that if $\mathcal{U}$ is a point-bounded family in $K$ which is meeting a set $A \subseteq X$, then there exists a point-bounded system of bumps of cardinality $\abs{\mathcal{U}}$ in $K$, which is meeting $A$.

We start with several auxiliary results. The following lemma is essentially known, not in this exact formulation, though. We present the proof for the convenience of the reader.

\begin{lemma}
\label{point finite}
Let $K_1, K_2$ be compact spaces and $E$ be a Banach space. Assume that $T: \C(K_1) \rightarrow \C(K_2, E)$ is an isomorphic embedding, and let $\mathcal{P}$ be a point-bounded system of bumps in $\C(K_1)$. For a fixed $\ep>0$ such that $\ep<\frac{1}{\norm{T^{-1}}}$ we consider the family of open sets
\[\mathcal{U}_{T, \ep}=\{y \in K_2: \norm{Tf(y)} > \ep\}_{f \in \mathcal{P}}.\]
Then:
\begin{itemize}
    \item[(i)] Assume that $E$ has nontrivial cotype. Then the family $\mathcal{U}_{T, \ep}$ is point-bounded. 
    \item[(ii)] Assume that $E$ does not contain an isomorphic copy of $c_0$. Then the family $\mathcal{U}_{T, \ep}$ is point-finite.
\end{itemize}
\end{lemma}

\begin{proof}
First we note that the assumption $\ep<\frac{1}{\norm{T^{-1}}}$ ensures that the sets forming the family $\mathcal{U}_{T, \ep}$ are nonempty.

Assume that $E$ has cotype $q \in [2, \infty)$. By the definition of cotype it follows that there exits a constant $C>0$ such that each $m \in \en$ and vectors $e_1, \ldots, e_m \in E$ satisfying $\norm{e_i} \geq \delta>0$ for each $i=1, \ldots, m$, there exist signs $\alpha_1, \ldots, \alpha_m \in S_{\er}$ such that $\norm{\sum_{i=1}^m \alpha_i e_i} \geq \delta C m^{\frac{1}{q}}$.

Thus, for the proof of (i), assume that there are distinct functions $f_1, \ldots, f_n \in \mathcal{P}$ and a point $y \in K_2$ such that for each $i=1, \ldots, n$, $\norm{Tf_i(y)} > \ep$. Then it is possible to find real numbers $\alpha_1, \ldots, \alpha_n \in S_{\er}$ such that $\norm{\sum_{i=1}^n \alpha_i Tf_i(y)} >\ep C n^{\frac{1}{q}}.$ On the other hand, since $\mathcal{P}$ is a point-bounded system of bumps, it is clear that $\norm{\sum_{i=1}^n \alpha_i f_i} \leq k$, where $k$ is the order of $\mathcal{P}$. Consequently, we have
\[\ep C n^{\frac{1}{q}} < \norm{\sum_{i=1}^n \alpha_i Tf_i(y)} = \norm{T(\sum_{i=1}^n \alpha_i f_i)(y)} \leq k\norm{T}.\]
It follows that $n < (\frac{k\norm{T}}{\ep C})^{q}$, and hence the system $\mathcal{U}_{T, \ep}$ is point-bounded.

For the proof of (ii) we want to show that for each $y \in K_2$, the set $\{f \in \mathcal{P}: \norm{Tf(y)}>\ep \}$ is finite. Thus we assume that there exist $y_0 \in K_2$ and a sequence $\{f_n\}_{n \in \en} \subseteq \mathcal{P}$ satisfying that $\norm{Tf_n(y_0)}> \ep$ for each $n \in \en$. Now, if we show that the series $\sum_{n=1}^{\infty} Tf_n(y_0)$ is weakly unconditionally Cauchy in $E$ (which means that $\sum_{n=1}^{\infty} \abs{\la e^*, Tf_n(y_0) \ra}< \infty$ for each $e^* \in \E$), we obtain a contradiction with the fact that $E$ does not contain an isomorphic copy of $c_0$ (see \cite[Theorem 6.7]{morrison2001functional}).
To show this, we consider the evaluation mapping $\phi\colon K_2 \times \E \to \C(K_2, E)^*$ defined as
	\[ \la \phi(y, \e), g \ra=\la \e, g(y) \ra, \quad g \in \C(K_2, E), y \in K_2, \e \in \E.\] 
It is clear that $\norm{\phi(y, \e)}=\norm{\e}$. Now, we fix $e^* \in E^*$, and let $T^*$ be the adjoint of $T$. Further, for a fixed $n \in \en$, let $\alpha_1, \ldots, \alpha_n \in S_{\er}$ satisfy 
	\[\abs{\la T^*\phi(y_0, e^*), f_i \ra}=\alpha_i \la T^*\phi(y_0, e^*), f_i \ra, \quad i=1, \ldots, n.\]
	Further, let $k$ be the order of $\mathcal{P}$. Then we have 
	\begin{equation}
	\nonumber
	\begin{aligned}
	&\sum_{i=1}^{n} \abs{\la e^*, Tf_i(y_0) \ra}=\sum_{i=1}^{n} \abs{\la \phi(y_0, e^*), Tf_i \ra}=
	\sum_{i=1}^{n} \abs{\la T^*\phi(y_0, e^*), f_i \ra}
	=\\&=
	\sum_{i=1}^{n} \alpha_i \la T^*\phi(y_0, e^*), f_i \ra=\la T^*\phi(y_0, e^*), \sum_{i=1}^{n} \alpha_i f_i \ra 
	\leq \\& \leq
	\norm{T^*\phi(y_0, e^*)} \norm{\sum_{i=1}^n \alpha_i f_i}\leq \norm{T^*} \norm{e^*}{\norm{\sum_{i=1}^n \alpha_i f_i}}  \leq  k\norm{T^*}\norm{e^*}. 
	\end{aligned}
	\end{equation}
	Thus also $\sum_{n=1}^{\infty} \abs{\la e^*, Tf_n(y_0) \ra} \leq k\norm{T^*}\norm{e^*}<\infty$. Hence the series $\sum_{n=1}^{\infty} Tf_n(y_0)$ is weakly unconditionally Cauchy, which finishes the proof.
\end{proof}

\begin{lemma}
\label{system}
Let $K_1, K_2$ be compact spaces, $L_2 \subseteq K_2$, $E$ be a Banach space, and assume that $T: \C(K_1) \rightarrow \C(K_2, E)$ is an isomorphic embedding. Let $\mathcal{P}$ be an infinite point-bounded system of bumps in $K_1$. Assume that at least one of the following conditions hold.
\begin{itemize}
\item[(i)] 
$E$ has nontrivial cotype and $\abs{\mathcal{P}}>c_{\text{pb}}(L_2, K_2)$,
\item[(ii)]
$E$ does not contain an isomorphic copy of $c_0$ and $\abs{\mathcal{P}}>c_{\text{pf}}(L_2, K_2)$. 
\end{itemize}
Then for each $\ep>0$ which is smaller than $\frac{1}{\norm{T^{-1}}}$,
\[ \abs{\{ f \in \mathcal{P}: \norm{Tf|_{L_2}} < \ep\}}=\abs{\mathcal{P}}. \]
\end{lemma}

\begin{proof}
For $f \in \mathcal{P}$ we consider the open set $U_{f, T, \frac{\ep}{2}}=\{y \in K_2: \norm{Tf(y)} > \frac{\ep}{2}\}$. By Lemma \ref{point finite}(i), the family $\mathcal{U}_{T, \frac{\ep}{2}}=\{ U_{f, T, \frac{\ep}{2}}:f \in \mathcal{P}\}$ of open sets is point-bounded if $E$ has nontrivial cotype, and it is point-finite if $E$ does not contain $c_0$ isomorphically. Thus in both cases, $\abs{\{f \in \mathcal{P}:U_{f, T, \frac{\ep}{2}} \cap L_2 \neq \emptyset \}} < \abs{\mathcal{P}}$. Hence, the cardinality of the set $\{f \in \mathcal{P}:U_{f, T, \frac{\ep}{2}} \cap L_2 = \emptyset \}$ is equal to cardinality of $\mathcal{P}$. This means that for $\abs{\mathcal{P}}$-many functions $f$ from $\mathcal{P}$, $\norm{Tf|_{L_2}} \leq \frac{\ep}{2}<\ep$.
\end{proof}

The following lemma serves as a tool for obtaining the estimate of distortion of the embedding $T$ in Theorem \ref{core}. For the proof see \cite[Lemma 2.2(b)]{rondos-somaglia}. 
\begin{lemma}
\label{norm}
Let $K_1, K_2$ be compact spaces, $E$ be a Banach space, and let $T:\C(K_1) \rightarrow \C(K_2, E)$ be an isomorphic embedding. Let $n, k \in \en$, $n>k$, and let $\ep>0$ be given. Suppose that there exist functions $g_1, \ldots, g_n \in \C(K_1)$ and $x \in K_1$ such  that $g_1(x)=\ldots =g_n(x)=1$, $0 \leq g_1 \leq \ldots \leq g_n \leq 1$, and such that the family
\[\{ y \in K_2: \norm{Tg_i(y)} \geq \ep \}_{i=1}^n\]
has order $k$. Then there exists a linear combination $f$ of the functions $g_1, \ldots, g_n$ such that $\norm{f}=1$ and
\[\norm{Tf} \geq \frac{2n-k}{k\norm{T^{-1}}}-\frac{2n-2k}{k}\ep.\]
\end{lemma}

We will also need to use the following result (see \cite[Proposition 2.4(a)]{rondos-somaglia}). We stress that here the sets $L_1, L_2$ are considered as topological spaces with the inherited topology.

\begin{prop}
\label{pom}
Let $K_1, K_2$ be compact spaces and $E$ be a Banach space not containing an isomorphic copy of $c_0$. Suppose that $T:\C(K_1) \rightarrow \C(K_2, E)$ is an isomorphic embedding, let $L_1 \subseteq K_1$, $U$ be an open set containing $L_1$, and let $L_2 \subseteq K_2$ be a compact set. If $\Gamma(ht(L_1))>\Gamma(ht(L_2))$, then for each $\ep>0$ there exist a function $f \in \C(K_1, [0, 1])$ and $x \in L_1$ such that $f=1$ on an open neighbourhood of $x$, $f=0$ on $K_1 \setminus U$ and $\norm{Tf|_{L_2}}<\ep$.
\end{prop}

Now we are ready to prove Theorem \ref{core}. 
Throughout the proof, for a function $g \in \C(K_1)$ and $\ep>0$ we will use $\{ \norm{Tg} \geq \ep \}$ as a shortcut for $\{ y \in K_2: \norm{Tg(y)} \geq \ep \}$.

\begin{proof}[Proof of Theorem \ref{core}.]
We will only prove that the lower bound $\frac{2n+2-k}{k}$ is true, since the boundedness by the number $3$ follows from Theorem \ref{dist3}.

Further, we prove just the item (i), the proof of (ii) is the same except that we would use condition (ii) in Lemma \ref{system} instead of condition (i). Thus we assume that $E$ has nontrivial cotype, $T:\C(K_1) \rightarrow \C(K_2, E)$ is an isomorphic embedding and $c(K_1^{(\omega^{\alpha}n)}, K_1)>c(K_2^{(\omega^{\alpha}k)}, K_2)$ for some ordinal $\alpha$ and $n, k \in \en$ such that $n \geq k$, and we fix a small enough $\ep>0$ so that the set $\{ \norm{Tg} \geq \ep \}$ is nonempty for each $g \in \C(K_1)$ of norm $1$. Our aim is to find a nonempty point-bounded system of bumps of height $n+1$ in $\C(K_1)$, which  satisfies that for each $(g_1, \ldots, g_{n+1})$ in it, the family \[\{y \in K_2: \norm{Tg_i(y)} \geq \ep \}_{i=1}^{n+1}\]
has order $k$. Once we find this system, it follows by Lemma \ref{norm} applied to an arbitrary member $(g_1, \ldots, g_{n+1})$ of this system that there exists a function $f \in \C(K_1)$ of norm $1$ such that 
	\[\norm{Tf} \geq \frac{2n+2-k}{k\norm{T^{-1}}}-\frac{2n+2-2k}{k}\ep.\] 
This, since $\ep>0$ was chosen arbitrarily, proves that $\norm{T}\norm{T^{-1}}\geq \frac{2n+2-k}{k}$. To find the desired system, it is enough to prove the following claim. 

\begin{claim}
	For each $i=1, \ldots, n+1$ there exist a point-bounded system of bumps $\mathcal{P}_i$ in $\C(K_1)$, of height $i$ and cardinality strictly greater than $c(K_2^{(\omega^{\alpha} k)}, K_2)$, satisfying that for each 
	$(g_1, \ldots ,g_{i}) \in \mathcal{P}_i$ and for each $j=1, \ldots, i$ the following assertions hold:
    \begin{itemize}
        \item[(i)] the intersection \[int(\{x \in K_1: g_j(x)=1\}) \cap K_1^{(\omega^{\alpha}(n-j+1))}\]
        is nonempty, and
        \item[(ii)] the set \[M_j^i(g_1, \ldots ,g_i)=K_2^{(\omega^{\alpha} \max\{k-j+1, 0\})} \cap \bigcup_{A\in[\{1,\ldots,i\}]^j} \bigcap_{p \in A} \{ \norm{Tg_p} \geq \ep\}\]
        is empty.
    \end{itemize}
Here $[\{ 1,\ldots,i\}]^j$ stands for the set of all subsets of $\{1,\dots,i\}$ of cardinality $j$.
\end{claim}

We note that the system that we were looking for above will be $\mathcal{P}_{n+1}$. To see that this system satisfies the above assumption, it is enough to notice that this follows from the fact that for each $(g_1, \ldots, g_{n+1}) \in \mathcal{P}_{n+1}$, the set $M_{k+1}^{n+1}(g_1, \ldots, g_{n+1})$ is empty. Thus the proof of the above claim will finish the proof of Theorem \ref{core}.

\emph{Proof of Claim.}
We proceed by finite induction.
First, since $c_{\text{pb}}(K_1^{(\omega^{\alpha}n)}, K_1)>c_{\text{pb}}(K_2^{(\omega^{\alpha}k)}, K_2)$, we may find a point-bounded system of bumps $\mathcal{P}$ in $\C(K_1)$ of cardinality strictly greater than $c(K_2^{(\omega^{\alpha} k)}, K_2)$, which is meeting $K_1^{(\omega^{\alpha}n)}$. Now, an application of Lemma \ref{system}(i) gives a subsystem $\mathcal{P}_1$ of $\mathcal{P}$ of the same cardinality, such that for each $f \in \mathcal{P}_1$, $\norm{Tf|_{K_2^{(\omega^{\alpha }k)}}}<\ep$, which finishes the case $i=1$. 

Now, let us assume that $1 \leq i<n+1$ and we suppose that we have found the system $\mathcal{P}_i$. First, for each $(g_1, \ldots, g_i) \in \mathcal{P}_i$ we find a function $g_{i+1}$ in the following way. We know that for each $j=1, \ldots, i$, the set 
\[M_j^i(g_1, \ldots ,g_i)=K_2^{(\omega^{\alpha} \max\{k-j+1, 0\})} \cap \bigcup_{A\in[\{ 1,\ldots,i\}]^j} \bigcap_{p \in A} \{ \norm{Tg_p} \geq \ep\}\]
is empty. Thus for each $j=1, \ldots, i$, if we denote
\[N_j^i(g_1, \ldots ,g_i)=K_2^{(\omega^{\alpha} \max\{k-j, 0\})} \cap \bigcup_{A\in[\{1,\dots,i\}]^j} \bigcap_{p \in A} \{ \norm{Tg_p} \geq \ep\},\]
then $ht(N_j^i(g_1, \ldots ,g_i)) \leq \omega^{\alpha}$. 
Further, we denote 
\[N(g_1, \ldots ,g_i)=\bigcup_{j=1}^{i} N_j^i.\] 
Then $N(g_1, \ldots ,g_i)$ is compact, and $ht(N(g_1, \ldots ,g_i)) \leq \omega^{\alpha}$ by \cite[Lemma 2.3(a)]{rondos-somaglia}. Thus $\Gamma(ht(N(g_1, \ldots ,g_i))) \leq \omega^{\alpha}$.

Further, we know that we can find an open set $V$ such that $g_i=1$ on $V$, and $V \cap K_1^{(\omega^{\alpha}(n-i+1))}$ is nonempty. Then, it holds $ht(V \cap K_1^{(\omega^{\alpha}(n-i))}) > \omega^{\alpha}$, see \cite[Lemma 2.3(b) and (c)]{rondos-somaglia}. Hence \[\Gamma(ht(V \cap K_1^{(\omega^{\alpha}(n-i))})>\omega^{\alpha}.\] Consequently, an application of Proposition \ref{pom} gives a function $g_{i+1} \in \C(K_1, [0, 1])$ and a point $x \in V \cap K_1^{(\omega^{\alpha}(n-i))}$ such that $g_{i+1}=1$ on an open neighbourhood of $x$, $g_{i+1}=0$ on $K_1 \setminus V$ and $\norm{Tg_{i+1}(y)}<\ep$ for each $y \in N(g_1, \ldots ,g_i)$. Then $g_{i+1}\leq g_i$. This finishes the the procedure of assigning to each $(g_1, \ldots ,g_i) \in \mathcal{P}_i$ the function $g_{i+1}$.

Now, we define the system $\mathcal{P}_{i+1}$ as
\[\mathcal{P}_{i+1}=\{(g_1, \ldots, g_{i+1}):(g_1, \ldots, g_{i}) \in \mathcal{P}_i \text{ and } \norm{T(g_{i+1})|_{K_2^{(\omega^{\alpha }k)}}} <\ep \}.\]
We need to check that the system $\mathcal{P}_{i+1}$ has the desired properties. Firstly, it is clear that  $\{g_{i+1}: (g_1, \ldots, g_i) \in \mathcal{P}_i\}$ is a point-bounded system of bumps in $\C(K_1)$. Hence, an application of Lemma \ref{system}(i) to this system shows that $\abs{\mathcal{P}_{i+1}}=\abs{\mathcal{P}_{i}}>\abs{c(K_2^{(\omega^{\alpha} k)}, K_2)}$. Further, from the construction it is clear that the item (i) is satisfied for $\mathcal{P}_{i+1}$. It remains to check that also the item (ii) is satisfied for $\mathcal{P}_{i+1}$, that is, we want to prove that for each $(g_1, \ldots, g_{i+1}) \in \mathcal{P}_{i+1}$ and for each $j=1, \ldots, i+1$, the set $M_j^{i+1}(g_1, \ldots, g_{i+1})$ is empty.

To this end, we fix $(g_1, \ldots, g_{i+1}) \in \mathcal{P}_{i+1}$, and we denote $N_0^{i}(g_1, \ldots, g_{i})=K_2^{(\omega^{\alpha} k)}$. We further fix $j \in \{0, \ldots, i\}$. Then we know that

\begin{equation}
\nonumber
\begin{aligned}
\emptyset=&
\{\norm{Tg_{i+1}} \geq \ep\} \cap N_j^i(g_0, \ldots ,g_i)
=\\&
\{\norm{Tg_{i+1}} \geq \ep\} \cap K_2^{(\omega^{\alpha} \max\{k-j, 0\})} \cap \bigcup_{A\in[\{ 1,\dots,i\}]^j} \bigcap_{p \in A} \{ \norm{Tg_p} \geq \ep\}
=\\& K_2^{(\omega^{\alpha} \max\{k-j, 0\})} \cap \bigcup_{A\in[\{1,\dots,i\}]^{j}} \bigcap_{p \in A\cup\{i+1\}} \{ \norm{Tg_p} \geq \ep\}
\end{aligned}
\end{equation}

Shifting the index $j$ by $1$, we obtain that for each $j=1, \ldots, i+1$,

\[\emptyset=K_2^{(\omega^{\alpha} \max\{k-j+1, 0\})} \cap \bigcup_{A\in[\{1,\dots,i\}]^{j-1}} \bigcap_{p \in A \cup \{i+1\}} \{ \norm{Tg_p} \geq \ep\}.\]

Further, it is simple to check that the set $M_j^{i+1}(g_1, \ldots, g_{i+1})$ can be written as a union of the above set and the set $M_{j}^{i}(g_1, \ldots, g_{i})$ for $j=1, \ldots, i+1$(if we use the convention that the set $M_{i+1}^{i}(g_1, \ldots, g_{i})$ is empty).
Thus we conclude that the set $M_j^{i+1}(g_1, \ldots, g_{i+1})$ is empty by the inductive assumption and the previous equality, which finishes the proof.
\end{proof}

The proof of Theorem \ref{main} now follows quite easily.

\begin{proof}[Proof of Theorem \ref{main}.]
We prove (i), the proof of (ii) is completely analogous. Thus we suppose that $E$ has nontrivial cotype, and let $T:\C(K_1) \rightarrow \C(K_2, E)$ is an isomorphic embedding. Further, we assume that $\alpha$ is an ordinal such that
\[\min_{\beta<\omega^{\alpha}} c_{\text{pb}}(K_1^{(\beta)}, K_1) > \min_{\beta<\omega^{\alpha}} c_{\text{pb}}(K_2^{(\beta)}, K_2)\]
and we seek a contradiction. We begin by noticing that if $\alpha$ is nonzero ordinal, then $\omega^{\alpha}$ is a limit ordinal, and it follows easily by compactness that the above minima are either $0$ or infinite.
Now, we distinguish $3$ cases. 

First we assume that $\alpha=0$. Our assumption then reads as $c(K_1)>c(K_2)$. Note that the cardinals $c(K_1), c(K_2)$ are not finite since we assume that $K_1, K_2$ are infinite. We find a cellular system of bumps $\mathcal{P}$ in $\C(K_1)$ of cardinality greater than $c(K_2)$. It then follows by Lemma \ref{system}(i) that for each $\ep>0$ there exists a function $f \in \mathcal{P}$ such that $\norm{Tf}<\ep$. But this contradicts the fact that $T$ is an isomorphic embedding.

Next we assume that $\alpha$ is a successor ordinal. Then, we find a positive integer $k$ satisfying that $c_{\text{pb}}(K_2^{(\omega^{\alpha-1}k)}, K_2)<\min_{\beta<\omega^{\alpha}} c_{\text{pb}}(K_1^{(\beta)}, K_1)$. Then for each $n \in \en$, $c_{\text{pb}}(K_1^{(\omega^{\alpha-1}n)}, K_1)>c_{\text{pb}}(K_2^{(\omega^{\alpha-1}k)}, K_2)$. Therefore, it follows from Theorem \ref{core}(i) that $\norm{T}\norm{T^{-1}} \geq \frac{2n+2-k}{k}$, which, since $n$ is arbitrary large, again contradicts the fact that $T$ is an isomorphic embedding. 

Finally, if $\alpha$ is a limit ordinal, we proceed similarly as in the successor case. Assuming that 
\[\min_{\beta<\omega^{\alpha}} c_{\text{pb}}(K_1^{(\beta)}, K_1) > \min_{\beta<\omega^{\alpha}} c_{\text{pb}}(K_2^{(\beta)}, K_2)\]
we may find an ordinal $\gamma<\alpha$ such that $c_{\text{pb}}(K_2^{(\omega^{\gamma})}, K_2)<\min_{\beta<\omega^{\alpha}} c_{\text{pb}}(K_1^{(\beta)}, K_1)$. Then for each $n \in \en$, $c_{\text{pb}}(K_1^{(\omega^{\gamma}n)}, K_2)>c_{\text{pb}}(K_2^{(\omega^{\gamma})}, K_2)$. Hence
$\norm{T}\norm{T^{-1}} \geq 2n+1$ for each $n \in \en$ by Proposition \ref{core}(i), which again leads to a contradiction. The proof is finished.
\end{proof}

\section{Preservation of spread}

 In this section we prove that isomorphic embeddings of $\C(K, E)$ spaces, with $E$ having nontrivial cotype, preserve the spread of $K$. First we note that since each subset $L$ of $K$ which has finite height is a union of finitely many subsets of $K$ of the form $L \setminus L^{(1)}, \ldots , L^{(ht(L)-2)} \setminus L^{(ht(L)-1)}, L^{(ht(L)-1)}$, where each of these sets with the inherited topology is discrete, it is clear that discrete sets can be replaced by sets of finite height in the definition of spread.  
 
 Further, if $K$ is compact and $x \in K$, let $\chi_{\{x\}}$ stands for the characteristic function of the point $x$. We recall that $\chi_{\{x\}}$ may be naturally viewed as an element of $\C(K)^{**}$ via the formula $\chi_{\{x\}}(\mu)=\mu(\{x\}), \mu \in \C(K)^*$. We start with a lemma.

\begin{lemma}
\label{discrete}
Let $K_1, K_2$ be compact Hausdorff spaces, $E$ be a Banach space, and let $T: \C(K_1) \rightarrow \C(K_2, E)$ be an isomorphic embedding. Then there exists an $\ep>0$, depending only on $T$, such that for each $x \in K_1$, the set 
\[Y_x=\{y \in K_2, \exists e^* \in S_{E^*}: \abs{\la T^{**}(\chi_{\{x\}}), \ep_y \otimes e^* \ra}>\ep \}\] 
is nonempty. Here $T^{**}:\C(K_1)^{**} \rightarrow \C(K_2, E)^{**}$ stands for the second adjoint of $T$, and $\ep_y \otimes e^*$ is an element of norm $1$ in $\C(K_2, E)^{*}$ defined for $f \in \C(K_2, E)$ by $\la \ep_y \otimes e^*, f \ra=\la e^*, f(y) \ra$.

Further, assume that $E$ has nontrivial cotype. Then:
\begin{itemize}
    \item[(i)] For each $y \in K_2$ there are at most finitely many $x \in K_1$ such that $y \in Y_x$.
    \item[(ii)] If $D$ is a discrete subset of $K_1$ and for each $x \in D$ we fix an arbitrary $y_x \in Y_x$, then the set $Y_D=\{y_x\}_{x \in D}$ has finite height.
\end{itemize}
\end{lemma}

\begin{proof}
It is well-known that the set $Y_x$ is nonempty if $\ep<\frac{1}{2\norm{T^{-1}}}$, see e.g. \cite[Proof of Theorem 1.4]{GalegoVillamizar}, or \cite[Lemma 4.2]{rondos-spurny-lattices} for a more general statement. Moreover, it is a standard fact that functions of the form $f_U$, where $U$ is an open neighbourhood of $x$ and $f_U$ is a function in $\C(K_1, [0, 1])$ which attains the value $1$ at $x$ and is $0$ on the complement of $U$, converge weak$^*$ to $\chi_{\{x\}}$ in $\C(K_1)^{**}$ with the natural ordering $U_1 \prec U_2$ iff $U_2 \subseteq U_1$ (see e.g. \cite[Lemma 4.1]{rondos-spurny-lattices}, where a slightly more general statement is proved). Thus, since $T^{**}$ is weak$^*$-weak$^*$ continuous, the functions $Tf_U$ converge weak$^*$ to $T^{**}(\chi_{\{x\}})$. 

From now on, we assume that the Banach space $E$ has nontrivial cotype. For the proof of (i), let $x_1, \ldots, x_n \in K_1$ and $y \in K_2$ be such that $y \in Y_{x_i}$ for $i=1, \ldots, n$. Thus there exist functionals $e_1^*, \ldots, e_n^* \in S_{E^*}$ such that $\abs{\la T^{**}(\chi_{\{x_i\}}), \ep_y \otimes e_i^* \ra}>\ep$ for $i=1, \ldots, n$. Now we find pairwise disjoint open sets $U_1, \ldots, U_n$ in $K_1$ such that each $U_i$ contains $x_i$, and functions $f_1, \ldots, f_n$ such that for each $i$, $spt f_i \subset U_i$ and $\abs{\la Tf_i, \ep_y \otimes e_i^* \ra}=\abs{\la e_i^*, Tf_i(y) \ra}>\ep$. Thus for each $i$, $\norm{Tf_i(y)} > \ep$. Hence it follows by Lemma \ref{point finite}(i) that there is an upper bound for such $n$, which finishes the proof of (i).

To prove $(ii)$, we assume that $D$ is a discrete subset of $K_1$ and the set $Y_D^{(n)}$ is nonempty for some $n \in \en$. We aim to find a finite cellular system of bumps $\mathcal{P}=\{f_i\}_{i=1}^{n}$ in $\C(K_1)$ and a finite sequence $\{y_i\}_{i=1}^{n}$ in $D$ such that for each $i \in \{1, \ldots, n\}$ and $j \leq i$, $\norm{Tf_i(y_j)}>\ep$. Once we find the system, we in particular obtain that for each $i=1, \ldots, n$, $\norm{Tf_i(y_1)}>\ep$. Thus again by Lemma \ref{point finite}(i) there is an upper bound for such $n$. 

To find the desired system, we proceed by finite induction. To start with, we pick $y_n \in Y_D^{(n)}$, and we find $x_n \in D$ such that $y_n \in Y_{x_n}$. Hence, we know that there exists a functional $e_n^*$ in the unit sphere of $E^*$ such that $\abs{\la T^{**}(\chi_{\{x_n\}}), \ep_{y_n} \otimes e_n^* \ra}>\ep$. Next, we may find a neighbourhood $U_n$ of $x_n$ such that $U_n \cap D=\{x_n\}$ and a function $f_n \in \C(K_1, [0, 1])$ whose support is contained in $U_n$, such that $\abs{\la Tf_n, \ep_{y_n} \otimes e_n^*\ra}>\ep$. Hence $\norm{Tf_n(y_n)}>\ep$.

Now, assume that $1 <i \leq n$, and that we have found distinct points $x_i, \ldots, x_n \in D$, $y_i, \ldots, y_n \in K_2$ such that $y_j \in Y_D^{(j)}$ for $j=i, \ldots n$, functions $f_i, \ldots f_n \in \C(K_1, [0, 1])$ whose supports are contained respectively in pairwise disjoint nonempty open sets $U_i, \ldots, U_n$, and such that $U_j \cap D=\{x_j\}$ and $\norm{Tf_j(y_i)} \geq \ep$ for each $j=i, \ldots, n$. Then, since $y_i \in Y_D^{(i)}$, we may find an $y_{i-1} \in Y_D^{(i-1)}$, which lies in the nonempty open set 
\[ \bigcap_{j=i}^n \{z \in K_2:\norm{Tf_i(z)}>\ep\}.\]
Further, we know that we can find some $x_{i-1} \in D$, distinct from $x_i, \ldots, x_n$, such that $y_{i-1} \in Y_{x_{i-1}}$.
Hence, there exists $e_{i-1}^* \in S_{E^*}$ such that 
\[\abs{\la T^{**}(\chi_{\{x_{i-1}\}}), \ep_{y_{i-1}} \otimes e_{i-1}^* \ra}>\ep.\]
Note that $x_{i-1}$ does not lie in the open set $U_i \cup \ldots \cup U_n$.
Consequently, we can find an open set $U_{i-1}$ that does not intersect $U_i, \ldots, U_n$, and such that $U_{i-1} \cap D=\{x_{i-1}\}$, and a function $f_{i-1} \in \C(K_1, [0, 1])$ whose support is contained in $U_{i-1}$, and such that $\abs{\la Tf_{i-1}, \ep_{y_{i-1}} \otimes e_{i-1}^*\ra}>\ep$. Thus $\norm{Tf_{j}(y_{i-1})}>\ep$ for each $j=i-1, i, \ldots, n$. This finishes the induction step and the proof.
\end{proof}

The proof of Theorem \ref{spread} now follows easily.

\begin{proof}[Proof of Theorem \ref{spread}.]
By the previous lemma, if $T:\C(K_1) \rightarrow \C(K_2, E)$ is an isomorphic embedding and $E$ has nontrivial cotype, we choose a small enough $\ep>0$. Then, if $D$ is and infinite discrete subset of $K_1$, for each $x \in D$ we fix an arbitrary point $y_x$ from the nonempty set $Y_x$, and we obtain set of cardinality $\abs{D}$ and finite height in $K_2$. Thus by the remark before the statement of Lemma \ref{discrete} it follows that $s(K_1) \leq s(K_2)$.
\end{proof}

\section{Isomorphisms with distortion less than $3$}

We recall that it has been proved in \cite{Gordon3} for the case or real-valued functions and extended in \cite{GalegoVillamizar} to the case of functions having values in a Banach space that does not contain $c_0$, that if $\C(K_1)$ is isomorphically embedded into $\C(K_2, E)$ by an isomorphism with distortion strictly less than $3$, then for each ordinal $\alpha$, the cardinality of the derivative $K_1^{(\alpha)}$ is less than or equal to the cardinality of $K_2^{(\alpha)}$. The fact that the Banach-Mazur distance between the spaces $\C([0, \omega])$ and $\C([0, \omega 2])$ is exactly $3$ (see \cite{Gordon3}) shows that the number $3$ is optimal for this result. In this section, we prove Theorem \ref{dist3}, which says that in the above case, there are more cardinal invariants of the derived sets that are preserved.

\begin{proof}[Proof of Theorem \ref{dist3}.]
For each $x \in K_1$ we consider the set 
\[\Lambda_x=\{y \in K_2: \text{ for each } f \in \C(K_1, [0, 1]) \text{ which satisfies } f(x)=1, \norm{Tf(y)}>\ep\},\]
where the $\ep>0$ is chosen small enough (depending only on $T$), such that for each ordinal $\alpha$, if $x \in K_1^{(\alpha)}$, the intersection $\Lambda_x \cap K_2^{(\alpha)}$ is nonempty, see \cite[Proof of Theorem 1.7]{GalegoVillamizar} or \cite[Lemma 3.1(b)]{rondos-scattered-subspaces}. Furthermore, for each $y \in K_2$, there are at most finitely many points $x \in K_1$ such that $y \in \Lambda_x$, see again \cite[Proof of Theorem 1.7]{GalegoVillamizar} or \cite[Lemma 3.1 (a)]{rondos-scattered-subspaces}.

Let $\mathcal{U}$ be a point-bounded family in $K_1$ which is meeting $K_1^{(\alpha)}$ for some ordinal $\alpha$. We find a point-bounded system of bumps $\mathcal{P}$ in $K_1$ of cardinality $\mathcal{U}$, which is meeting $K_1^{(\alpha)}$.  Then by Lemma \ref{point finite}, the family
\[\mathcal{U}_{T, \ep}=\{y \in K_2: \norm{Tf(y)} > \ep\}_{f \in \mathcal{P}}\]
is point-bounded if $E$ has nontrivial cotype, and it is point-finite if $E$ does not contain a copy of $c_0$. Moreover, we know from above that the family $\mathcal{U}_{T, \ep}$ is meeting $K_2^{(\alpha)}$, which proves the respective inequalities $c_{\text{pb}}(K_1^{(\alpha)}, K_1) \leq c_{\text{pb}}(K_2^{(\alpha)}, K_2)$ and $c_{\text{pb}}(K_1^{(\alpha)}, K_1) \leq c_{\text{pf}}(K_2^{(\alpha)}, K_2)$.

For the proof of the part concerning the spread, we fix an ordinal $\alpha$, and let $D$ be an infinite discrete subset of $K_1^{(\alpha)}$. For each $x \in D$ we pick an $y_x \in \Lambda_x \cap K_2^{(\alpha)}$, and we claim that the set $\Lambda_D=\{y_x\}_{x \in D}$ has finite height. The proof of this is basically the same as in the proof of Lemma \ref{discrete}, just a bit simpler. 
Thus we assume that the set $\Lambda_D^{(n)}$ is nonempty for some $n \in \en$.  As before, we want to find a finite cellular system of bumps $\mathcal{P}=\{f_i\}_{i=1}^{n}$ in $\C(K_1)$ and a finite sequence $\{y_i\}_{i=1}^{n}$ in $D$ such that for each $i \in \{1, \ldots, n\}$ and $j \leq i$, $\norm{Tf_i(y_j)}>\ep$. First, we pick $y_n \in \Lambda_D^{(n)}$, and we find $x_n \in D$ such that $y_n \in \Lambda_{x_n}$. We also find a neighbourhood $U_n$ of $x_n$ such that $U_n \cap D=\{x_n\}$ and a function $f_n \in \C(K_1, [0, 1])$ whose support is contained in $U_n$, such that $f_n(x_n)=1$. Hence $\norm{Tf_n(y_n)}>\ep$.

Now, assume that $1 <i \leq n$, and that we have found distinct points $x_i, \ldots, x_n \in D$, $y_i, \ldots, y_n \in K_2$ such that $y_j \in \Lambda_D^{(j)}$ for $j=i, \ldots n$, functions $f_i, \ldots f_n \in \C(K_1, [0, 1])$ whose supports are contained respectively in pairwise disjoint nonempty open sets $U_i, \ldots, U_n$, and such that $U_j \cap D=\{x_j\}$ and $\norm{Tf_j(y_i)} > \ep$ for each $j=i, \ldots, n$. Then, since $y_i \in \Lambda_D^{(i)}$, we may find an $y_{i-1} \in \Lambda_D^{(i-1)}$, which lies in the nonempty open set 
\[ \bigcap_{j=i}^n \{z \in K_2:\norm{Tf_i(z)}>\ep\}.\]
Further, we know that we can find some $x_{i-1} \in D$, distinct from $x_i, \ldots, x_n$, such that $y_{i-1} \in \Lambda_{x_{i-1}}$. Since $x_{i-1}$ does not lie in the open set $U_i \cup \ldots \cup U_n$, we can find an open set $U_{i-1}$ that does not intersect $U_i, \ldots, U_n$, $U_{i-1} \cap D=\{x_{i-1}\}$, and a function $f_{i-1} \in \C(K_1, [0, 1])$ whose support is contained in $U_{i-1}$, and such that $f_{i-1}(x_{i-1})=1$. Thus $\norm{Tf_{j}(y_{i-1})}>\ep$ for each $j=i-1, i, \ldots, n$. This finishes the induction step.

Now, since we in particular know that $\norm{Tf_i(y_1)}>\ep$ for each $i=1, \ldots, n$, by Lemma \ref{point finite}(i), there is an upper bound for such $n$, which proves that $\Lambda_D$ has finite height. Moreover, by above, for each $y \in K_2$, there are at most finitely many points $x \in K_1$ such that $y \in \Lambda_x$. Thus the cardinality of $\Lambda_D$ is equal to the cardinality of $D$. Since $\Lambda_D \subset K_2^{(\alpha)}$, the proof is finished. 
\end{proof}

\noindent\textbf{Acknowledgments.}\enspace The author would like express his gratitude to professors Antonio Avilles, Piotr Koszmider, Witold Marciszewski and Gregorz Plebanek for sharing their knowledge about the isomorphic theory of spaces of continuous functions with the author. Many thanks to prof. Ond{\v{r}}ej Kalenda for his useful advice and comments on the content of the paper.


\bibliography{iso-functions}\bibliographystyle{siam}
\end{document}